\documentclass[12pt,leqno]{amsart}
\newtheorem{theorem}{Theorem}[section]
\newtheorem{definition}{Definition}[section]

\newtheorem{proposition}{Proposition}[section]
\newtheorem{remark}{Remark}[section]

\newtheorem{example}{Example}[section]

\newcommand{\be}{\begin{equation}}
\newcommand{\ee}{\end{equation}}
\numberwithin{equation}{section}
\newcommand{\bea}{\begin{eqnarray}}
\newcommand{\eea}{\end{eqnarray}}
\newcommand{\beb}{\begin{eqnarray*}}
\newcommand{\eeb}{\end{eqnarray*}}
\usepackage{amssymb,amsfonts,amsthm,setspace,indentfirst}
\usepackage[dvips]{graphics}
\usepackage{epsfig}
\begin{document}
\title{On $\mathcal{I}^{st}$- convergence in the space of reals}
\author{Amar Kumar banerjee$^{1}$ and Khairul Hasan$^{2}$}
\address{$^{1}$,$^{2}$ Department of Mathematics, The University of Burdwan, Golapbag, Burdwan-713104, West Bengal, India.} 
\email{$^{1}$akbanerjee@math.buruniv.ac.in, akbanerjee1971@gmail.com}
\email{$^{2}$khairul9734@gmail.com}
\begin{abstract}
 In this paper, we study the notion of $\mathcal{I}$$^{st}$-convergence of sequences of real numbers which is an extension work of $\mathcal{I}$$^*$-convergence on real sequences. We also define the convergence field of $\mathcal{I}$$^{st}$-convergence and  $\mathcal{I}$$^{st}$-limit points and we discuss some properties related to these notions.
\end{abstract}
\maketitle
\section{\bf{Introduction }}
The concept of $\mathcal{I}^*$-convergence came out from a result \cite{TS} on statistical convergence of real sequences. The result is as follows:
A real sequence $\{x_n\}$ is statistically convergent to $\eta$ if and only if there exists a set $M=\{m_1<m_2<...<m_k<...\}$ such that $d(M)=1$ and $\displaystyle \lim_kx_{m_k}=\eta$, i.e., the subsequence $\{x_n\}_{n\in M}$ is convergent to $\eta$ in ordinary sense. So, the questions naturally arise what would happen if we consider the idea of statistical convergence of the subsequence $\{x_{m_k}\}$ (i.e., $\{x_n\}_{n\in M}$) of $\{x_n\}$ instead of the ordinary convergence of $\{x_{m_k}\}$ and this has led to investigate the occurrence of new phenomena on convergence. Keeping in mind, we have used the concept of statistical limit of the subsequence $\{x_n\}_{n\in M}$ in the definition of $\mathcal{I^*}$-convergence and then we have observed what changes have been come out in the results of $\mathcal{I^*}$-convergence. With this motion, we have moved our attention towards this direction. In this paper, we have studied the notion of $\mathcal{I}^{st}$-convergence as a parallel notion of $\mathcal{I}^*$-convergence. We have introduced the idea of $\mathcal{I}^{st}$-convergence using the concept of statistical convergence and $\mathcal{I}^*$-convergence. Let us briefly discuss these two types of convergence.\\
The concept of statistical convergence was given by H.Fast \cite{HF} and was later developed by many authors (see \cite{{AS},{FRIDY},{PK1},{KS},{TS},{HS}} etc.). Actually, the idea of statistical convergence is a generalization of the usual convergence of a real sequence. Now we recall the definition of the natural density of a set $B\subseteq \mathbb{N}$ where $\mathbb{N}$ denotes the set of natural numbers. Let $B_n=\{b\in B: b\leq n\}$ and $|B_n|$ denotes the cardinality of $B_n$. The natural density of $B$ is defined by $d(B)=\displaystyle \lim_{k\to \infty}\frac{|B_n|}{n}$ if the limit exists. A real sequence $\{x_n\}$ is said to be statistical convergence to $\eta$ if for every $\epsilon>0$, $d(A(\epsilon))=0$ where $A(\epsilon)=\{n\in \mathbb{N}:|x_n-\eta|\geq \epsilon\}$. It is clear that every ordinary convergence is always statistical convergence, but the converse may not be true.\\
After long years,the notion of statistical convergence has been extended to $\mathcal{I}$, $\mathcal{I^*}$-convergence by Kostyko et al. and also to strong-$\mathcal{I^K}$-convergence \cite{{AM2}}, $\mathcal{I}$-divergence and $\mathcal{I^*}$-divergence \cite{{AA}} etc. using the concept of ideals of subsets of natural numbers. A real sequence $\{x_n\}$ is said to be $\mathcal{I^*}$-convergent to $\eta$ if there exists a set $M=\{m_1<m_2<...<m_k<...\}\in \mathcal{F(I)}$ ( $\mathcal{F(I)}$ is the filter associated with the ideal $\mathcal{I}$ ) such that $\displaystyle \lim_k x_{m_k}=\eta$. It should be mentioned that, if $\{x_n\}$ is $\mathcal{I^*}$-convergent to $\eta$ then it is $\mathcal{I}$-convergent, but converse may not be true. However, the converse holds if an additional condition, known as `AP condition', holds which has been given in Remark 4.1. We have divided this paper into five sections. The concept of $\mathcal{I}^{st}$-convergence of real sequence is introduced and several fundamental results are studied. The necessary and sufficient condition for the equivalence of $\mathcal{I}$-convergence and $\mathcal{I}^{st}$-convergence is also established. The idea of a convergence field of $\mathcal{I}^{st}$-convergence has been given and we have also studied the notion of $\mathcal{I}^{st}$-limit points and found some relevant properties.

\section{\bf{Preliminaries}}
Let $X$ be a non-empty set. We denote $2^X$ as the power set of $X$.
Then a family of sets $\mathcal{I} \subset 2^X$ is said to be an ideal if 
(i) $A,B \in \mathcal{I} \Rightarrow A \cup B \in \mathcal{I}$ and
(ii) $A \in \mathcal{I}, B \subset A \Rightarrow B \in \mathcal{I}$. In view of (ii), the empty set $\phi \in \mathcal{I}$. \\ 
$\mathcal{I}$ is called a non-trivial ideal if $\mathcal{I} \neq 2^X,\{\phi\}$. A non-trivial ideal $\mathcal{I}$ in $X$ is called admissible if $\{x\} \in \mathcal{I}$ for each $x \in X$. Clearly, the family $\mathcal{F(I)}=\{A \subset X: X \setminus A \in \mathcal{I}\}$ is a filter on $X$ which is called the filter associated with $\mathcal{I}$. We note that the set of all finite subsets of $\mathbb{N}$ is an ideal which we denoted by $\mathcal{I}_f$ and the set of all subsets of $\mathbb{N}$ whose density are zero is also an ideal which we denoted by $\mathcal{I}_d$.

Let $\mathcal{I}$ be a non-trivial ideal of $\mathbb N$. A sequence $\{x_n\}$ in $\mathbb R$ is said to be $\mathcal{I}$-convergent to $x$ if for every $\varepsilon>0$, the set $A(\varepsilon)=\{n \in \mathbb N: |x_n-x| \geq \varepsilon\} \in \mathcal{I}$. \\
If $\{x_n\}$ is $\mathcal{I}$-convergent to $x$, then $x$ is called $\mathcal{I}$-limit of $\{x_n\}$ and we write $\mathcal{I}-\displaystyle \lim_nx_n=x$.

Let $\mathcal{I}$ be an admissible ideal in $\mathbb N$. A sequence $\{x_n\}$ of real numbers is said to be $\mathcal{I^*}$-convergent to $x$ (shortly $\mathcal{I^*}-\displaystyle \lim_n x_n=x$) if there is a set $M=\{m_1 < m_2 < ... \}\in \mathcal {F(I)}$ such that $\displaystyle \lim_{k\to\infty}x_{m_k}=x$. Clearly, if $\{x_n\}$ is ordinary convergent to $x$ then it is $\mathcal{I}_f$-convergent to $x$ and if $\{x_n\}$ is statistically convergent to $x$ then it is $\mathcal{I}_d$-convergent to $x$.

\section{\bf{$\mathcal{I}^{st}$-convergence}}
\begin{definition}(cf\cite{{SD}})
 Let $\mathcal{I}$ be an admissible ideal in $\mathbb N$. A sequence $\{x_n\}$ of real numbers is said to be $\mathcal{I}^{st}$-convergent to $x$, written as $\mathcal{I}^{st}-\displaystyle \lim_n x_n=x$, if there is a set $M=\{m_1 < m_2 < ... \}\in \mathcal {F(I)}$ such that $ st-\displaystyle \lim_{k\to\infty}x_{m_k}=x$.
\end{definition}

\begin{remark}
   The idea of $\mathcal{I}^*$-statistical convergence in a metric space given by Debnath and Rakshit \cite{SD} is as follows: A sequence $x=\{x_n\}_{n\in\mathbb{N}}$ of elements of $X$ is said to be $\mathcal{I}^*$-statistical convergent to $\xi\in X$ if and only if there is a set $M=\{m_1 < m_2 < ...<m_k<... \}\in \mathcal {F(I)}$ such that $ st-\displaystyle \lim_{k\to\infty}d(x_{m_k},\xi)=0$. In this paper our definition of $\mathcal{I}^{st}$-convergence is the same as that of $\mathcal{I}^*$-statistical convergence \cite{SD}. But our approach, direction, and mode of work are different from that of \cite{SD}. Here we have used the idea of statistical convergence of the subsequence $\{x_{n_k}\}$ of $\{x_n\}$ in place of ordinary convergence of $\{x_{n_k}\}$ in the definition of $\mathcal{I}^*$-convergence. As per our motivation the notation $\mathcal{I}^{st}$-convergence seems to be more appropriate  than $\mathcal{I}^*$- statistical convergence. That's why we have used different notation `$\mathcal{I}^{st}$-convergence' in place of `$\mathcal{I}^*$-statistical convergence'.
\end{remark}

\begin{theorem}
    For any sequence $\{x_n\}$, $\mathcal{I}^{st}-\displaystyle \lim_n x_n$ is unique.
\end{theorem}
\begin{proof}
If possible, let $\mathcal{I}^{st}-\displaystyle \lim_n x_n=x$ and $\mathcal{I}^{st}-\displaystyle \lim_n x_n=y$ such that $x\neq y$. Let $\varepsilon=\frac{|x-y|}{3}>0$. Then there exist $M_1$, $M_2\in \mathcal {F(I)}$ such that $\mathcal{I}^{st}-\displaystyle \lim_{n\in M_1} x_n=x$ and $\mathcal{I}^{st}-\displaystyle \lim_{n\in M_2} x_n=y$.\\  Then, $M=M_1 \cap M_2(\neq \phi)\in \mathcal{F(I)}$. Also, $d(\{n\in M_1: |x_n-x|\geq \varepsilon\})=0$ and $d(\{n\in M_2: |x_n-y|\geq \varepsilon\})=0$.  Consider $P=\{n\in M: |x_n-x|< \varepsilon\}$ and $Q=\{n\in M: |x_n-y|< \varepsilon\}$. From statistical convergence on $M_1$ and $M_2$, we get $d(P)=1$ and $d(Q)=1$ (since $M\subseteq M_1, M_2$). But $P\cap Q=\phi$. Therefore, $d(P\cup Q)=d(P)+d(Q)=2$, a contradiction. Hence $x=y$. This completes the proof.
\end{proof}
\begin{theorem}
For any sequence $\{x_n\}$ of real numbers, $\mathcal{I^*}-\displaystyle \lim_n x_n=x$ implies $\mathcal{I}^{st}-\displaystyle \lim_n x_n=x$.
\end{theorem}
\begin{proof} 
 Since every ordinary convergent sequence is statistically convergent, by definition, $\displaystyle \lim_{k\to\infty}x_{m_k}=x$ implies $ st-\displaystyle \lim_{k\to\infty}x_{m_k}=x$.
 Hence $\mathcal{I}^{st}-\displaystyle \lim_n x_n=x$.
\end{proof}

\begin{remark}
    The converse of the above theorem may not be true, which is shown in the following example.
\end{remark}

\begin{example}
Let $\mathbb{N}={\bigcup_{j=1}^\infty}D_j$ be a decomposition of $\mathbb{N}$ (i.e., $D_i \cap D_j=\phi$ for $i\neq j$), where $D_j(j=1,2,...)$ are infinite sets and $D_j=\{2^{j-1}(2s-1):s=1,2,...\}$. Denote by $\mathcal{J}$ the class of all $A\subseteq \mathbb{N}$ such that $A$ intersects only a finite number of $D_j$. Then $\mathcal{J}$ is an admissible ideal in $\mathbb{N}$.\\
Let $k_0$ be a fixed positive integer and $M=\{n\in \mathbb{N}: n\in\ D_j ,n\neq k^2 , k>k_0\}$ , $H=\{n\in\mathbb{N}: n\in\  D_j,n\neq k^2 , k\leq k_0\}$, $F=\{ n\in \mathbb{N}:n\in\  D_j ,n=k^2\}$.  Define a sequence $\{x_n\}_{n\in\mathbb{N}}$ as follows:
\begin{equation*}
    x_n=
    \begin{cases}
        \frac{1}{j}, &  n\in\ M \\
        2j+1,  &   n\in\ H  \\
        2j,  &  n\in\ F  
    \end{cases}
\end{equation*}
Clearly $M\cup F \in \mathcal{F(\mathcal{J})}$ and $\{x_n\}_{n\in M\cup F}$ is a subsequence of $\{x_n\}_{n\in\mathbb{N}}$ which is statistically convergent to 0. But $\mathcal{I}^*-\displaystyle \lim_n x_n=0$ does not hold.
\end{example}

\begin{definition}
    Let $X\neq \phi$. A non-trivial ideal $\mathcal{I}$ in $X$ is called d-admissible ideal if $\mathcal{I}$ contains every subset of $X$ whose density is zero.
\end{definition}

\begin{theorem}
    Let $\mathcal{I}$ be a d-admissible ideal. Then $\mathcal{I}^{st}-\displaystyle \lim_n x_n=x$ implies $\mathcal{I}-\displaystyle \lim_n x_n=x$.
\end{theorem}

\begin{proof}
 Let $\varepsilon >0$ and $\{x_n\}$ be a sequence such that $\mathcal{I}^{st}-\displaystyle \lim_n x_n=x$. Then there exists $M=\{n_1<n_2<...\}\in \mathcal {F(I)}$ such that $st-\displaystyle \lim_{k\to\infty}x_{n_k}=x$.\\
 Therefore, $d(\{k\in\mathbb{N}: |x_{n_k}-x|\geq \varepsilon\})=0$ $\Rightarrow d(A)=0$,where $A=\{k\in\mathbb{N}: |x_{n_k}-x|\geq\varepsilon\}$.\\
 Let $B=\{n_k:k\in A\}\subset M$.
 Now $A(\varepsilon)=\{n\in \mathbb{N}; |x_n-x|\geq \varepsilon\} \subseteq (\mathbb{N}\backslash M) \cup B$.\\ Therefore, $A(\varepsilon)\in \mathcal{I}$, as $\mathbb{N}\backslash M \in \mathcal{I}$ and $d(B)=0$. Hence, $\mathcal{I}-\displaystyle \lim_n x_n=x$ 
\end{proof}

\begin{remark}
    The converse of the above theorem may not be true. This is shown by the following example.
\end{remark}
\begin{example}
Take $\mathcal{I}=\mathcal{J}$ ( $\mathcal{J}$ is mentioned in Example 3.1 ). Consider a sequence $\{x_n\}$  defined by $x_n=\frac{1}{j}$ for $n\in D_j(j=1,2,...)$. Clearly $\mathcal{I}-\displaystyle \lim_n x_n=0$.
If possible, let $\mathcal{I}^{st}-\displaystyle \lim_n x_n=0$. Then there exists  $M=\{n_1<n_2<...\}\in \mathcal {F(I)}$ such that $st-\displaystyle \lim_{k\to\infty}x_{n_k}=0$. So for every $\varepsilon>0$, $d(A(\varepsilon))=0$ where
\begin{equation}
   A(\varepsilon)=\{k\in \mathbb{N}:|x_n-0|\geq\varepsilon\} 
\end{equation}
So $H=\mathbb{N}\backslash M\in \mathcal{I}$. So there exists $p\in \mathbb{N}$ such that $H\subset D_1\cup D_2 \cup ...\cup D_p$ and $D_{p+1}\subset \mathbb{N}\backslash H=M$. So there are infinite many $k's$ such that $x_{n_k}=\frac{1}{p+1}$.\\
 Let $\varepsilon=\frac{1}{p+2}$. Then $B(\varepsilon)=\{n_k\in M: |x_{n_k}-0|>\frac{1}{p+2}\}=D_{p+1}$.\\ But $d(D_{p+1})\neq0$. This contradicts (3.1). So $\mathcal{I}^{st}-\displaystyle \lim_n x_n=0$ is not true.
\end{example}

\section{\bf{Convergence fields of $\mathcal{I}^*$-convergence and $\mathcal{I}^{st}$-convergence}}
Let $l_\infty$ be the set of all real bounded sequences. Then $l_\infty$ becomes a normed linear space with respect to the sup-norm defined by $\|x\|=\displaystyle\sup_n|x_n|$, $x=\{x_n\}\in l_\infty$.
\begin{definition}\cite{PK1}
The convergence fields of $\mathcal{I}$-convergence and $\mathcal{I}^*$-convergence are denoted, respectively, by $F(\mathcal{I})$, $F(\mathcal{I}^{*})$ and are defined as follows: \\
$F(\mathcal{I})=\{x=\{x_n\}\in l_\infty : \mathcal{I}-\displaystyle \lim_n x_n\in \mathbb{R} \ \ exists\}$ \\$F(\mathcal{I}^{*})=\{x=\{x_n\}\in l_\infty : \mathcal{I}^{*}-\displaystyle \lim_n x_n \in \mathbb{R} \ \ exists\}$.
\end{definition}
\begin{definition}
The convergence field of $\mathcal{I}^{st}$ can be denoted by $F(\mathcal{I}^{st})$ and is defined by $F(\mathcal{I}^{st})=\{x=\{x_n\} \in l_\infty : \mathcal{I}^{st}-\displaystyle \lim_n x_n \in \mathbb{R}\ \ exists \}$.
\end{definition}
Then clearly $F(\mathcal{I})$ , $F(\mathcal{I}^*)$ and $F(\mathcal{I}^{st})$ are subsets of the normed linear space $l_\infty$.

From Theorem 3.1 and Theorem 3.2 we can summarize our results regarding the convergence fields $F(\mathcal{I})$, $F(\mathcal{I}^*)$ and $F(\mathcal{I}^{st})$. From Theorem 3.1, we know that $F(\mathcal{I}^{*})\subseteq F(\mathcal{I}^{st})$ and from Theorem 3.2, we get $F(\mathcal{I}^{st}) \subseteq F(\mathcal{I})$. Thus $F(\mathcal{I}^*)\subseteq F(\mathcal{I}^{st})\subseteq F(\mathcal{I})$.

\begin{definition}\cite{AM1}
Let $\mathcal{I}$and $\mathcal{K}$ be two non-trivial ideals on a non empty subset $S$ of $\mathbb{N}$, the set of natural numbers. For the ideals $\mathcal{I}$and $\mathcal{K}$ on $S$,
\begin{center}
    $\mathcal{I} \vee \mathcal{K}=\{A\cup B:A\in \mathcal{I},B\in \mathcal{K}\}$
\end{center} is an ideal which is the smallest ideal containing both $\mathcal{I}$and $\mathcal{K}$ on $S$, i.e., $\mathcal{I},\mathcal{K}\subseteq \mathcal{I} \vee \mathcal{K}$. 
\end{definition}

\begin{theorem}
If $\mathcal{I}$ is an admissible ideal then $F(\mathcal{I}^{st})$ is a linear subspace of $l_\infty$.
\end{theorem}

\begin{proof}
Let $\{x_n\},\{y_n\} \in F(\mathcal{I}^{st})$. Then $\mathcal{I}^{st}-\displaystyle \lim_n x_n=x_0$ and $\mathcal{I}^{st}-\displaystyle \lim_n y_n=y_0$ for some $x_0,y_0 \in \mathbb{R}$.\\ So there exist $M_1,M_2 \in F(\mathcal{I)}$ such that $\{x_n\}_{n\in M_1}$ is statistically convergent to $x_0$ and $\{y_n\}_{n\in M_2}$ is statistically convergent to $y_0$. Let $M=M_1\cap M_2\in F(\mathcal{I)}$. Clearly $M$ is infinite, since $\mathcal{I}$ is non-trivial. Also $M\subseteq M_1,M_2$. Now we claim that $\{x_n\}_{n\in M}$ is statistically convergent to $x_0$.\\
Let $\varepsilon>0$. Since $\{x_n\}_{n\in M_1}$ is statistically convergent to $x_0$, so $d(\{k\in M_1 : |x_k-x_0|\geq \varepsilon\})=0$.\\ Now $\{k\in M : |x_k-x_0|\geq \varepsilon\}\subseteq \{k\in M_1 : |x_k-x_0|\geq \varepsilon\}$. Therefore, $d(\{k\in M : |x_k-x_0|\geq \varepsilon\})=0$. Hence $\{x_n\}_{n\in M}$ is statistically convergent to $x_0$. Similarly, $\{y_n\}_{n\in M}$ is statistically convergent to $y_0$. Statistical convergence is linear, so $\{ax_n+by_n\}_{n\in M}$ is statistically convergent to $ax_0+by_0$. Now we have found a $M$ in $F(\mathcal{I)}$ such that $\{ax_n+by_n\}_{n\in M}$ is statistically convergent to $ax_0+by_0$. Therefore, $\{ax_n+by_n\}$ is $\mathcal{I}^{st}$-convergent and $\mathcal{I}^{st}-\displaystyle \lim_n\{ax_n+by_n\}=ax_0+by_0$. Hence $F(\mathcal{I}^{st})$ is a linear subspace of $l_\infty$.
\end{proof}
\begin{theorem}
    Suppose $\mathcal{I}$ is an admissible ideal in $\mathbb{N}$ and $\mathcal{I}_d=\{A\subseteq \mathbb{N} : d(A)=0\}$. Then $F(\mathcal{I} \ \vee \mathcal{I}_d)$ is a closed subset of $l_\infty$.
\end{theorem}

\begin{proof}
Let $y^{(m)}=\{\{{y_j}^{(m)}\}_{j=1}^\infty\}\subset F(\mathcal{I} \ \vee \mathcal{I}_d) \  (m=1,2,...)$ such that $lim_{m\to \infty}y^{(m)}=y$ where $y=\{y_j\}_{j=1}^\infty \in l_\infty$. So, $\displaystyle \lim_{m\to\infty}\|y^{(m)}-y\|=0$.\\ Now we show that $y\in F(\mathcal{I} \ \vee \mathcal{I}_d)$. Based on the assumption that 
\begin{equation}
    (\mathcal{I} \ \vee \mathcal{I}_d)-\displaystyle \lim_j\ y_j^{(m)}=\xi_m\in \mathbb{R} \ \textit{exists for each } m=1,2,....
\end{equation}
We first show that $\{\xi_m\}$ is a Cauchy sequence and then secondly we show that $(\mathcal{I} \ \vee \mathcal{I}_d)-\displaystyle \lim_j y=\xi$.

 So, let $\varepsilon>0$. Since $\displaystyle \lim_{m\to \infty}y^{(m)}=y$, we say that $\{y^{(m)}\}_{m=1}^\infty$ is a Cauchy sequence in $l_\infty$. So, there exists an element $n_0\in \mathbb{N}$ such that for each $u,v>n_0$ we get
\begin{equation}
    \|y^{(u)}-y^{(v)}\|<\frac{\varepsilon}{3}
\end{equation}
Take two fixed natural numbers $u,v$ so that $u,v>n_0$. Then, by (4.1), the sets $U(\frac{\varepsilon}{3})=\{j:|y_j^{(u)}-\xi_u|<\frac{\varepsilon}{3}\}$, $V(\frac{\varepsilon}{3})=\{j:|y_j^{(v)}-\xi_v|<\frac{\varepsilon}{3}\} \in F(\mathcal{I} \ \vee \mathcal{I}_d)$. So, $ U(\frac{\varepsilon}{3}) \cap V(\frac{\varepsilon}{3})\neq \phi$. Let $t\in U(\frac{\varepsilon}{3}) \cap V(\frac{\varepsilon}{3})$. Then, 
\begin{equation}
    |y_t^{(u)}-\xi_u|<\frac{\varepsilon}{3}, |y_t^{(v)}-\xi_v|<\frac{\varepsilon}{3}
\end{equation}
Therefore,
    $|\xi_u-\xi_v|\leq |\xi_u-y_t^{(u)}|+|y_t^{(u)}-y_t^{(v)}|+|y_t^{(v)}-\xi_v|<\frac{\varepsilon}{3}+\frac{\varepsilon}{3}+\frac{\varepsilon}{3}=\varepsilon$, by (4.2) and (4.3).
Therefore, $\{\xi_m\}_1^\infty$ is a Cauchy sequence and so $\displaystyle \lim_{m\to \infty}\xi_m$ exists and is equal to $\xi$ (say)$ \in \mathbb{R}$.\\
Next, let $\eta>0$. Choose a natural number $\lambda_0$ such that
 \begin{equation}
     |\xi_\lambda-\xi|<\frac{\eta}{3}
 \end{equation} and
 \begin{equation}
     \|y^{(\lambda)}-y\|\leq\frac{\eta}{3}
 \end{equation}
hold together for $\lambda>\lambda_0$.
 So,\begin{equation}
     |y_n^{(\lambda)}-y_n|\leq\frac{\eta}{3} \ \text{for each } n.
 \end{equation}
 Now, for each $n\in\mathbb{N}$, we have 
 \begin{equation}
     |y_n-\xi|\leq|y_n-y_n^{(\lambda)}|+|y_n^{(\lambda)}-\xi_\lambda|+|\xi_\lambda-\xi|.
 \end{equation}
 Let $A(\eta)=\{n:|y_n-\xi|\geq\eta\}$ and $A_\lambda(\frac{\eta}{3})=\{n: |y_n^{(\lambda)}-\xi_\lambda|\geq\frac{\eta}{3}\}$.\\ So, if $n\in A^c_\lambda(\frac{\eta}{3})$, then by (4.4), (4.6) and (4.7), we get $|y_n-\xi|<\eta$ for every $n\in A^c_\lambda(\frac{\eta}{3})$. So we have, \begin{equation}
     A^c_\lambda(\frac{\eta}{3})\subseteq A^c(\eta).
 \end{equation}
 Since $A_\lambda(\frac{\eta}{3})\in \mathcal{I} \vee\mathcal{I}_d$, in view of (4.8), we find that $A(\eta)\in \mathcal{I} \vee\mathcal{I}_d$. This completes the proof.
\end{proof}

\begin{theorem}
For every admissible ideal $\mathcal{I}$ in $\mathbb{N}$ we have $\overline{F(\mathcal{I}^{st})}=F( \mathcal{I} \vee\mathcal{I}_d)$.
\end{theorem}

\begin{proof}
By Theorem 3.2 we have $F(\mathcal{I}^{st})\subseteq F(\mathcal{I} \vee\mathcal{I}_d)$. Since $F(\mathcal{I} \vee\mathcal{I}_d)$ is closed in $l_\infty$, we get $\overline{F(\mathcal{I}^{st})}\subseteq F(\mathcal{I} \vee\mathcal{I}_d)$.Therefore, this is enough to prove that $F(\mathcal{I} \vee\mathcal{I}_d)\subseteq \overline{F(\mathcal{I}^{st})}$.\\
For $z\in l_\infty $ and $\delta>0$, we denote the ball whose center at $z$ and radius $\delta$ by \begin{equation*}
    B(z,\delta)=\{x\in l_\infty : \|x-z\|<\delta\}.
\end{equation*}
Let $y=\{y_n\}\in F(\mathcal{I} \vee\mathcal{I}_d)$ be an arbitrary element and $0<\delta<1$. It is required to show that 
\begin{equation}
    B(y,\delta) \cap F(\mathcal{I}^{st})\neq \phi
\end{equation}
Put $L=\mathcal{I} \vee\mathcal{I}_d-\displaystyle \lim y$. Choose an arbitrary $\varepsilon$ such that $0<\varepsilon<\delta$. Then $A(\varepsilon)=\{n: |y_n-L|\geq \varepsilon\}\in \mathcal{I} \vee\mathcal{I}_d$. So, let $A(\varepsilon)=G\cup H$, where $G\in \mathcal{I}, H\in \mathcal{I}_d$.  Define a sequence $x=\{x_n\}$ such that $x_n=y_n$ if $n\in A(\varepsilon)$ and $x_n=L$ if $n\notin A(\varepsilon)$(i.e. if $n\in (\mathbb{N}\backslash G) \cap (\mathbb{N}\backslash H))$. Then $x \in l_\infty$. Since $G\in \mathcal{I}$, the set $\mathbb{N}\backslash G \in F(\mathcal{I})$. Now we consider the subsequence $\{x_n\}_{n\in\mathbb{N}\backslash G}$ of $\{x_n\}$ where
\begin{equation*}
\{x_n\}_{n\in\mathbb{N}\backslash G}=
    \begin{cases}
    L,& n\in (\mathbb{N}\backslash G) \cap (\mathbb{N}\backslash H)\\
    y_n, & n\in H\backslash (G\cap H)
    \end{cases}
\end{equation*}
Then $A(\varepsilon)=\{n\in \mathbb{N}\backslash G : |x_n-L|\geq \varepsilon\}\subset H\backslash (G\cap H)\subset H$. Now  $d(A(\varepsilon))=0$, as $d(H)=0$. Therefore $st-\displaystyle \lim_n \{x_n\}_{n\in\mathbb{N}\backslash G}=L$ and hence $x\in B(y,\varepsilon)$. So, (4.9) holds. This completes the proof.
\end{proof}

\begin{remark}
   Recall that an admissible ideal is said to satisfy the AP condition if for every countable family $\{A_1, A_2,...\}$ of mutually disjoint sets of $\mathcal{I}$ there exists a family $\{B_1,B_2,...\}\subset \mathcal{I}$ such that $A_j \triangle B_j$ is finite and $\bigcup_{j=1}^\infty B_j \in  \mathcal{I}$. This definition has been modified in the following way, replacing the condition `$A_j \triangle B_j$ is finite' by  $d(A_j \triangle B_j)=0$. This modified version of the AP condition is useful for proving the converse part of Theorem 4.4.
\end{remark}

\begin{definition}
An admissible ideal $\mathcal{I}$ in $\mathbb{N}$ is said to satisfy the condition AP(D) if for every countable system $\{A_1,A_2,...\}$ of mutually disjoint sets belonging to $\mathcal{I}$, there exist sets $B_j\subset  \mathcal{I} \ (j=1,2,...)$ such that $d(A_j \triangle B_j)=0$ and $B=\bigcup_{j=1}^\infty B_j \in  \mathcal{I}$, where $\triangle$ denotes the symmetric differences.
\end{definition}

\begin{theorem}
$\mathcal{I}-\displaystyle \lim_nx_n=\xi \Rightarrow \mathcal{I}^{st}-\displaystyle \lim_n x_n=\xi$ if and only if  $\mathcal{I}$ satisfies the AP(D) condition.
\end{theorem}

\begin{proof}
Suppose  $\mathcal{I}$ satisfies the AP(D) condition and  $\mathcal{I}-\displaystyle \lim_n x_n=\xi$. Then for every $\varepsilon>0$ the set $A(\varepsilon)=\{n\in \mathbb{N}: |x_n-\xi|\geq \varepsilon \}\in \mathcal{I}$. So, the sets $A_j\in \mathcal{I} \ (j=1,2,...)$ where 
\begin{center}
$A_1=A(1)=\{n\in \mathbb{N}: |x_n-\xi|\geq 1\}$\\
$A_k=A(\frac{1}{k})\backslash A(\frac{1}{k-1})=\{n\in \mathbb{N}: \frac{1}{k} \leq |x_n-\xi|<\frac{1}{k-1}$, $k=2,3,...$
\end{center}
clearly $A_i \cap A_j=\phi$ for all $i,j=1,2,...$ and $i\neq j$. Since $\mathcal{I}$ satisfies the AP(D) condition, so there exist sets  $B_j\subset  \mathcal{I} \ (j=1,2,...)$ such that $d(A_j\triangle B_j)=0$ for all $j=1,2,...$ and $\bigcup_{j=1}^\infty B_j=B \in  \mathcal{I}$. So $M=\mathbb{N}\backslash B \in \mathcal{F(I)}$. We show, for any $\varepsilon>0, d(\{n\in \mathbb{N}\backslash B :|x_n-\xi|\geq \varepsilon \})=0$.\\ Let $\varepsilon>0$. Then choose a number $k\in \mathbb{N}$ such that $\frac{1}{k+1}<\varepsilon$. Then $\{n: |x_n-\xi|\geq \varepsilon\}\subset \bigcup_{j=1}^{k+1} A_j$. The set $\bigcup_{j=1}^{k+1} A_j \in \mathcal{I}$. Let $\bigcup_{j=1}^{k+1} (A_j \triangle B_j)=A $. Then $d(A)=0$ and $\bigcup_{j=1}^{k+1} B_j \cap A^c=\bigcup_{j=1}^{k+1} A_j \cap A^c$. Let $n\in \mathbb{N}\backslash B$. Suppose that $n\in A^c$. Then $n\notin \bigcup_{j=1}^{k+1} B_j$ so $n\notin \bigcup_{j=1}^{k+1} A_j$. So $|x_n-\xi|<\frac{1}{k+1}<\varepsilon$.\\ So $\{n\in \mathbb{N}\backslash B :|x_n-\xi|\geq \varepsilon\}\subset A \Rightarrow d(\{n\in \mathbb{N}\backslash B :|x_n-\xi|\geq \varepsilon \})=0$.

Conversely, suppose that $\mathcal{I}-\displaystyle \lim_nx_n=\xi \Rightarrow \mathcal{I}^{st}-\displaystyle \lim_n x_n=\xi$. We will show that $\mathcal{I}$ satisfies the AP(D) condition. Let $\{A_1,A_2,...\}$ be a family of mutually disjoint sets of $\mathcal{I}$. Now we define a sequence $x=\{x_n\}$ as follows 
\begin{equation*}
    \begin{split}
      x_n=&\frac{1}{j},n\in A_j \ (j=1,2,...)\\
       =&0, n \in n\in \mathbb{N} \backslash \cup_j A_j
    \end{split}
\end{equation*}
First, we will show $\mathcal{I}-\displaystyle \lim_n x_n=0$. Let $\varepsilon>0$.
Choose an $p$ such that $\frac{1}{p}<\varepsilon$. Then $A(\varepsilon)=\{n\in \mathbb{N}: |x_n-\xi|\geq \varepsilon\}\subseteq A_1\cup A_2\cup...\cup A_p$. It is clear that $A(\varepsilon)\in \mathcal{I}$ and hence $\mathcal{I}-\displaystyle \lim_nx_n=0$. Therefore, by our assumption, we have $\mathcal{I}^{st}-\displaystyle \lim_n x_n=0$. So there exists a set $H\in \mathcal{F(I)}$ such that 
\begin{equation}
    st-\displaystyle \lim_{n\to \infty,n\in H} x_n=0
\end{equation}
Then $B=\mathbb{N}\backslash H \in \mathcal{I}$ and let $B_j=A_j\cap B,\ (j=1,2,...)$. It is enough to show that $d(A_j \triangle B_j)=0$ for $(j=1,2,...)$.
\begin{center}
Now $\bigcup_{j=1}^\infty B_j=\bigcup_{j=1}^\infty (B\cap A_j)=B\cap \bigcup_{j=1}^\infty A_j\subseteq B$.
\end{center}
Therefore, $\bigcup_{j=1}^\infty B_j \in \mathcal{I} $ as $B\in \mathcal{I}$. Put $H=\{l_1<l_2<...)$. Then from (4.7) we have $st-\displaystyle \lim_{k\to \infty} x_{l_k}=0$. It is clear that $A_j \triangle B_j=A_j\cap (\mathbb{N}\backslash B)$. \\ Let $A(\frac{1}{j})=\{l_k\in H :|x_{l_k}|\geq \frac{1}{j}\}$, $j=1,2,...$. Then for each $j$, $d(A(\frac{1}{j})=0$. Now if $m_k\in A_j \cap H$ then $x_{m_k}=\frac{1}{j}$. So $m_k \in \{l_k \in H:|x_{l_k}|\geq \frac{1}{j+1}\}$. Hence,
\begin{equation*}
 \begin{split}
    & A_j \cap H \subset \{l_k \in H:|x_{l_k}|\geq \frac{1}{j+1}\}\\
    &\Rightarrow d(A_j \cap H) \subset d(\{l_k \in H:|x_{l_k}|\geq \frac{1}{j+1}\})=0\\
    &\Rightarrow d(A_j \cap H)=0.  \ \text{This is true for each j}.
 \end{split}   
\end{equation*}
Therefore, $d(A_j \triangle B_j)=0$ as $d(A_j\cap  H)=0$.
\end{proof}

\section{\bf{$\mathcal{I}^{st}$-limit points}}
Recall that a number $\lambda$ is said to be statistical limit point of a sequence $\{x_n\}_{n\in \mathbb{N}}$ of real numbers if there exists a set $M=\{m_1<m_2<...\}\subset \mathbb{N}$ such that $d(M)\neq 0$ and $\displaystyle \lim_{k\to\infty}x_{m_k}=\lambda$. A number $\lambda\in \mathbb{R}$ is said to be a statistical cluster point of $x=\{x_n\}_{n\in \mathbb{N}}$ if for each $\varepsilon>0$ we have $d\{n\in \mathbb{N}: |x_n-\lambda|<\varepsilon\}\neq 0$ (see \cite{{JJ},{FRIDY2}}).\\
Denote by $\Lambda_x$ and $\Gamma_x$ the set of all statistical limit points and statistical cluster points of $x$, respectively.
\begin{definition}\cite{KS}
 Let $(X,\rho)$ be a metric space, $x=\{x_n\}_{n\in \mathbb{N}}$ a sequence of elements of $X$.\\
 (a) An element $\xi \in X$ is said to be an $\mathcal{I}$-limit point of $x$ if there is a set $M=\{m_1<m_2<...\}\subset \mathbb{N}$ such that $M\notin \mathcal{I}$ and $\displaystyle \lim_{k\to\infty}x_{m_k}=\xi$.\\
 (b) An element $\xi \in X$ is said to be an $\mathcal{I}$-cluster point of $x$ if and only if for each $\varepsilon>0$ we have $\{n\in \mathbb{N}:\rho(x_n,\xi)<\varepsilon\}\notin \mathcal{I}$.
 \end{definition}
Denote by $\mathcal{I}(\Lambda_x)$ and $\mathcal{I}(\Gamma_x)$ the set of all $\mathcal{I}$-limit and $\mathcal{I}$-cluster points of $x$, respectively.\\
We can extend these concepts to $\mathcal{I}^{st}$-convergence in the following way.
\begin{definition}
Let $(X,\rho)$ be a metric space and a sequence $x=\{x_n\}_{n\in \mathbb{N}}$ of $X$. An element $\xi \in X$ is said to be an $\mathcal{I}^{st}$-limit point of $x$ if there is a set $M=\{m_1<m_2<...\}\subset \mathbb{N}$ such that $M\notin \mathcal{I}$ and $ st-\displaystyle \lim_{k\to\infty}x_{m_k}=\xi$. 
\end{definition}
Notation: Denote by $\mathcal{I}^{st}(\Lambda_x)$ the set of all $\mathcal{I}^{st}$-limit points of $x$.\\
 We now give an example of a sequence that has $\mathcal{I}^{st}$-limit point but the sequence is not $\mathcal{I}^{st}$-convergent to any real number.
\begin{example}
 Let $\mathcal{I}=\mathcal{I}_d=\{A\subseteq \mathbb{N}: d(A)=0\}$ and the filter associated with $\mathcal{I}, \mathcal{F}(\mathcal{I}_d)=\{M \subseteq\mathbb{N}: d(\mathbb{N}\backslash M)=0\}=\{M: d(M)=1\}$. Define a sequence as follows:\\
  \begin{equation*}
    x_n=
    \begin{cases}
        7, &  n\in\ M=\{1,3,5,...\} \\
        n,  &   n\notin\ M    
    \end{cases}
\end{equation*}
So, $M\notin \mathcal{I}$ as $d(M)=\frac{1}{2}>0$. The subsequence $\{x_{m_k}\}$ for $m_k\in M$, $x_{m_k}=7$ for all $k$. Hence $st-\displaystyle \lim_{k\to \infty}x_{m_k}=7$. But the whole sequence $\{x_n\}$ has no $\mathcal{I}^{st}$-limit. For this, if possible, let there is a subset $M=\{m_1<m_2<...\}\in \mathcal{F}(\mathcal{I}_d)$ such that $st-\displaystyle \lim_{k\to \infty}x_{m_k}=\xi$. Therefore, $d(M)=1$ and $d(\mathbb{N}\backslash M)=0$. Let $A=\{2,4,6,...\}$. Then $d(A)=\frac{1}{2}$ and $d(M \cap A)=d(A)-d(A\backslash M)=\frac{1}{2}$ [as $A\backslash M \subseteq \mathbb{N}\backslash M$]. Thus, $M$ contains infinitely many even numbers and, in fact, the set of indices $k$ for which $m_k$ is even has positive lower density in the enumeration of $M$.\\
Case-1: Let $\xi\neq 7$. Then $|7-\xi|=\delta >0$. For every odd $m_k \in M$, we have $x_{m_k}=7$, so $|x_{m_k}-\xi|\geq \delta$. Since $d(M)=1$, the set $B=\{m_k\in M: m_k \ is \ odd\}$ has density $\frac{1}{2}$. Since the set $\{k: |x_{m_k}-\xi|\geq \delta\}$ has positive density and $\{k:|x_{m_k}-\xi|\geq \delta\}\supseteq B$, so $st-\displaystyle \lim_{k\to \infty}x_{m_k}=\xi$ can not hold.\\
Case-2: Let $\xi=7$. Take $\delta=1$. Now $x_{m_k}\geq 2$ for any even $m_k$. So for $m_k\geq 8$ we have $|x_{m_k}-7|\geq 1$. Since $m_k \to \infty$ and the set $A=\{m_k: m_k \ is \ even\}$ has positive density, so by similar argument to that above it follows that the set $\{k:|x_{m_k}-7|\geq 1\}$ has positive density and therefore $st-\displaystyle \lim_{k\to \infty}x_{m_k}=\xi$ can not hold. Thus there is no real number $\xi$ such that $st-\displaystyle \lim_{k\to \infty}x_{m_k}=\xi$ holds.
\end{example}

\begin{proposition}
 Let $\mathcal{I}$ be an admissible ideal.Then for each sequence $x=\{x_n\}_{n\in \mathbb{N}}$ of elements of $X$ we have $\mathcal{I}(\Lambda_x)\subset \mathcal{I}^{st}(\Lambda_x)$.
\end{proposition}
\begin{proof}
Let $\xi \in\mathcal{I}(\Lambda_x)$. So there exists a set $M=\{m_1<m_2<...\}\subset \mathbb{N}$ such that $M\notin \mathcal{I}$ and $\displaystyle \lim_{k\to\infty}\rho(x_{m_k},\xi)=0$.  This implies that $M\notin \mathcal{I}$ and $st-\displaystyle \lim_{k\to\infty}\rho(x_{m_k},\xi)=0 $. Hence $\xi \in \mathcal{I}^{st}(\Lambda_x)$.
Therefore, $\mathcal{I}(\Lambda_x)\subset \mathcal{I}^{st}(\Lambda_x)$.
\end{proof}

\begin{proposition}
 Let $\mathcal{I}$ be a d-admissible ideal.Then for each sequence $x=\{x_n\}_{n\in \mathbb{N}}$ of elements of $X$ we have $\mathcal{I}^{st}(\Lambda_x)\subset \mathcal{I}(\Gamma_x)$.
\end{proposition}
\begin{proof}
Let $\xi \in\mathcal{I}^{st}(\Lambda_x)$. Then there exists a set $M=\{m_1<m_2<...\}\notin \mathcal{I}$ such that 
\begin{equation*}
    \begin{split}
        & st-\displaystyle \lim_{k\to\infty}\rho(x_{m_k},\xi)=0\\
        & \Rightarrow d(\{k\in \mathbb{N}:\rho(x_{m_k},\xi)\geq \varepsilon\})=0\\
        & \Rightarrow d(B)=0, where \ B=\{k\in \mathbb{N}: \rho(x_{m_k},\xi)\geq \varepsilon\}.
        \end{split}
\end{equation*}
Take $\delta>0$. Then $\{n\in \mathbb{N}: \rho(x_n,\xi)<\delta\}\supset M/B$. Therefore, $\{n\in \mathbb{N}: \rho(x_n,\xi)<\delta\}\notin \mathcal{I}$, as $M/B \notin \mathcal{I}$. Hence $\xi \in\mathcal{I}(\Gamma_x)$. This completes the proof.
\end{proof}

\end{document}